\documentclass[pdflatex,sn-mathphys-ay]{sn-jnl}

\usepackage{amsmath,amssymb,amsfonts}
\usepackage{amsthm}
\usepackage{mathtools}
\usepackage{mathrsfs}
\usepackage{bm}
\usepackage{graphicx}
\usepackage{booktabs}
\usepackage{multirow}
\usepackage{tikz}
\usetikzlibrary{arrows.meta,positioning,calc}

\usepackage[nameinlink,capitalize,noabbrev]{cleveref}

\theoremstyle{plain}
\newtheorem{theorem}{Theorem}
\newtheorem{lemma}[theorem]{Lemma}
\newtheorem{proposition}[theorem]{Proposition}
\newtheorem{corollary}[theorem]{Corollary}

\newtheorem{conjecture}[theorem]{Conjecture}

\theoremstyle{remark}
\newtheorem{remark}[theorem]{Remark}
\newtheorem{example}[theorem]{Example}

\theoremstyle{definition}

\crefname{theorem}{theorem}{theorems}
\Crefname{theorem}{Theorem}{Theorems}
\crefname{lemma}{lemma}{lemmas}
\Crefname{lemma}{Lemma}{Lemmas}

\definecolor{colOne}{HTML}{55A8F5}
\definecolor{colTwo}{HTML}{68CD83}
\definecolor{colThree}{HTML}{F57888}
\definecolor{colOther}{HTML}{F3CC63}
\definecolor{edgeGray}{HTML}{536170}

\newcommand{\coregroup}[7]{%
  \begin{scope}[shift={(#2,#3)}]
    \fill[#4] (0,0) circle[radius=.64];
    \fill[#5]
      (0,0) -- (90:.64)
      arc[start angle=90,end angle=-90,radius=.64]
      -- cycle;
    \draw[white,line width=1pt] (0,-.64)--(0,.64);
    \node[
      circle,draw=edgeGray,line width=.9pt,
      minimum size=1.28cm,inner sep=0pt
    ] (#1) at (0,0) {};
    \node[font=\large] at (0,0) {$#6\rightarrow #7$};
  \end{scope}
}

\begin{document}

\title[Social Optimality Does Not Imply Coalitional Stability in Unweighted Max-\(k\)-Cut Games]{
Social Optimality Does Not Imply Coalitional Stability in Unweighted Max-\(k\)-Cut Games
}

\author*[1]{\fnm{Hritiz} \sur{Gogoi}}
\email{hgogoi@iitb.ac.in}

\affil*[1]{
  \orgdiv{Industrial Engineering and Operations Research},
  \orgname{Indian Institute of Technology Bombay},
  \orgaddress{
    \city{Mumbai},
    \postcode{400076},
    \state{Maharashtra},
    \country{India}
  }
}

\abstract{
\begin{abstract}

In Max-$k$-Cut games, optimal colorings maximize social
welfare, but are they stable against coordinated deviations
that strictly benefit every participating player?
A recurring conjecture in the literature asserts that,
for unweighted graphs, every optimal coloring has this
stability property, known as strong equilibrium.
We disprove the conjecture for every $k\ge3$ by exhibiting
a finite, simple, connected, unweighted graph with an
optimal coloring that is not a strong equilibrium.
The profitable coalitional deviation leads to another
optimal coloring: every coalition member gains, while
players outside the coalition bear the corresponding
losses and social welfare remains unchanged.
We also give an exact identity relating coalitional gains,
changes in cut value, and internal conflicts, recovering
the known stability guarantee for optimal colorings
when the maximum degree is at most $2k-1$.
For strong deviations from general Nash equilibria,
we establish a sharp upper bound on the number of
monochromatic edges remaining within the coalition
and show that this number can exceed the coalition size.
\end{abstract}
}

\keywords{
Max-\(k\)-Cut games,
Strong equilibria,
Coalitional deviations,
Optimal colorings,
Graph partitioning
}

\maketitle
\section{Introduction}
\label{sec:intro}

Can a socially optimal outcome be destabilized by a coalition whose
members all gain? The Max-$k$-Cut game provides a natural setting for
this question. Players are vertices of a graph, each chooses one of
$k$ colors, and each benefits from neighbors choosing different colors.
Consider, for example, transmitters selecting among $k$ communication
channels. An edge joins two transmitters that interfere when assigned
the same channel. In a simple unweighted model, each transmitter seeks
to minimize the number of interfering neighbors on its channel,
equivalently maximizing the number assigned different channels.
A channel assignment is therefore a \emph{coloring}, and the
\emph{utility} of each transmitter is the number of its neighbors
assigned a different channel. An edge whose endpoints have different
colors is called \emph{bichromatic} and represents a pair of
transmitters whose interference is avoided. Since each such edge
contributes one unit to the utilities of both endpoints, maximizing
total utility, or \emph{social welfare}, is equivalent to finding a
\emph{maximum $k$-cut} or equivalently, an optimal coloring: 
a coloring with the largest possible number of bichromatic edges.

At an optimal coloring, no transmitter can benefit by changing its
channel alone, since such an improvement would also increase the
number of bichromatic edges, contradicting optimality. Thus every
optimal coloring is also a \emph{Nash equilibrium}, an outcome in which
no player can gain by changing its choice alone. Transmitters may,
however, coordinate their channel changes, potentially benefiting
every participant at the expense of those outside the coordinating group.
Stability against such coordinated changes is more demanding:
a \emph{strong Nash equilibrium}, or simply \emph{strong equilibrium},
is an outcome in which no group of players, called a \emph{coalition},
can change their choices jointly so that every member strictly gains,
while the choices of players outside the coalition remain
fixed~\citep{aumann1959acceptable}.

Whether every optimal coloring is a strong equilibrium has been a
central open problem in the study of unweighted Max-$k$-Cut
games~\citep{carosi2019coalition,garuglieri2024optimal}.
We answer this question negatively for every $k\ge3$, even on
connected graphs. The counterexample preserves the cut value:
coalition members gain by eliminating internal conflicts, while
players outside the coalition bear the corresponding losses.

Formally, let $G=(V,E)$ be a finite simple graph and let
$k\ge2$ be an integer. Write $n:=|V|$, let $N(v)$ be the neighborhood of $v$,
and put $d(v):=|N(v)|$ and $\Delta(G):=\max_{v\in V}d(v)$.
The \emph{unweighted Max-$k$-Cut game}
on $(G,k)$ has vertex set $V$ as its set of players and
$[k]:=\{1,\dots,k\}$ as the strategy set of every player.
A \emph{coloring} is a map $\sigma\colon V\to[k]$, and the
utility of player $v$ under $\sigma$ is
\[
  u_v(\sigma)
  :=
  \bigl|\{u\in N(v):\sigma(u)\ne\sigma(v)\}\bigr|.
\]
Given a coloring, an edge is \emph{monochromatic} if its endpoints receive
the same color, and \emph{bichromatic} otherwise. The cut value is
\[
  \mathcal{C}(\sigma)
  :=
  \bigl|\{\{u,v\}\in E:\sigma(u)\ne\sigma(v)\}\bigr|,
\]
and the social welfare is $\sum_{v\in V}u_v(\sigma)=2\,\mathcal{C}(\sigma)$.
A coloring $\sigma^{\star}$ is \emph{optimal} if it maximizes social
welfare, equivalently the cut value, i.e.,
\[
  \sigma^{\star}\in
  \operatorname*{arg\,max}_{\sigma\colon V\to[k]}
  \sum_{v\in V}u_v(\sigma).
\]

For later use, define the number of neighbors of $v$ with color $c$ by
\[
  m_v(c;\rho):=|\{w\in N(v):\rho(w)=c\}|,
\]
and write $M(\rho):=|E|-\mathcal C(\rho)$ for the total number of
monochromatic edges under a coloring $\rho$.

A coloring is a \emph{Nash equilibrium} if no player can strictly
increase its utility by a unilateral recoloring. A \emph{strong
deviation} from $\sigma$ is a pair $(S,\tau)$, where
$\varnothing\ne S\subseteq V$, $\tau$ agrees with $\sigma$ outside
$S$, and every $v\in S$ changes color and satisfies
$u_v(\tau)>u_v(\sigma)$. Requiring every member to change color is
without loss of generality: members who do not move can be omitted.
A coloring is a \emph{strong equilibrium} if it admits no strong
deviation, and an \emph{$r$-strong equilibrium} if it admits none
with $|S|\le r$.

A unilateral recoloring changes the cut value by exactly the change
in the moving player's utility. This proves the Nash stability of
optimal colorings. For a coalition, however, an internal edge enters
the utilities of two members, whereas a boundary edge enters only
one. Thus a strict gain for each member need not increase the cut.
For $k=2$, every optimal coloring is nevertheless a strong
equilibrium, even in weighted graphs~\citep{gourves2009strong}.
For weighted graphs this implication fails when
$k\ge3$~\citep{gourves2010max}. The corresponding unweighted question
was formulated as the following conjecture
\citep{carosi2019coalition,garuglieri2024optimal}.

\begin{conjecture}\label{conj:optimal-strong}
For every finite simple unweighted graph $G$ and every
integer $k\ge3$, every optimal coloring of the Max-$k$-Cut
game on $(G,k)$ is a strong equilibrium.
\end{conjecture}

The conjecture would guarantee both the existence of a strong
equilibrium and the compatibility of coalitional stability with
maximum social welfare. Several partial results support it:
optimal colorings are $5$-strong equilibria
\citep{carosi2019coalition}, and subsequently were shown to be
$7$-strong equilibria~\citep{garuglieri2024optimal}.
Monochromatic coalitions cannot strongly deviate from an optimal
coloring~\citep{madeo2022game}. Table~\ref{tab:known} summarizes
these results and the contrast with the weighted setting. \\

\begin{table}[htbp]
\centering
\caption{Selected results on the stability of optimal colorings
in Max-$k$-Cut games. Here $\sigma^\star$ denotes an optimal
coloring, $n=|V(G)|$, and
$\Delta(G)=\max_{v\in V(G)}d(v)$ is the maximum vertex degree.
SE denotes strong equilibrium and $r$-SE denotes
$r$-strong equilibrium. A monochromatic coalition consists
of vertices having the same color under $\sigma^\star$.}
\label{tab:known}
\small
\setlength{\tabcolsep}{3pt}
\renewcommand{\arraystretch}{1.15}
\begin{tabular}{@{}p{.18\linewidth}p{.46\linewidth}p{.30\linewidth}@{}}
\toprule
Condition & Result & Reference \\
\midrule

\multicolumn{3}{@{}l}{\emph{Weighted graphs}} \\
\addlinespace[2pt]
$k=2$
& Every $\sigma^\star$ is an SE
& \cite{gourves2009strong} \\

$k\ge2$
& Every $\sigma^\star$ is a $3$-SE
& \cite{gourves2009strong,gourves2010max} \\

$k\ge3$
& An optimal coloring need not be a $4$-SE
& \cite{gourves2010max} \\

$k\ge n-2$
& Every $\sigma^\star$ is an SE
& \cite{gourves2010max} \\

\addlinespace[5pt]
\multicolumn{3}{@{}l}{\emph{Unweighted graphs}} \\
\addlinespace[2pt]
$k=2$
& Every $\sigma^\star$ is an SE
& \cite{gourves2009strong} \\

$k\ge2$
& Every $\sigma^\star$ is a $5$-SE
& \cite{carosi2019coalition} \\

$k\ge2$
& Every $\sigma^\star$ is a $7$-SE
& \cite{garuglieri2024optimal} \\

$k\ge2$,

& Every $\sigma^\star$ is an SE if $\Delta(G)\le2k-1$
& \cite{carosi2019coalition} \\

$k\ge2$
& No monochromatic coalition admits a strong
  deviation from $\sigma^\star$
& \cite{madeo2022game} \\

$k\ge3$
& An optimal coloring need not be an SE
& Theorem~\ref{thm:optimal-not-strong}\\
\bottomrule
\end{tabular}
\end{table}

Despite the stability guarantees above, optimal colorings
need not be stable against coalitions of arbitrary size.
Our main result disproves
Conjecture~\ref{conj:optimal-strong} for every $k\ge3$.

\begin{theorem}
\label{thm:optimal-not-strong}
For every integer $k\ge3$, there exists a finite simple,
connected, unweighted graph that admits an optimal
$k$-coloring which is not a strong equilibrium.
\end{theorem}

In our example, the coalition deviates to another optimal
coloring: every member strictly gains, although total
social welfare remains unchanged.
Section~\ref{sec:counterexample} presents the construction
and establishes its optimality through a combinatorial
count, completing the proof of
Theorem~\ref{thm:optimal-not-strong}.
Motivated by this example,
Section~\ref{sec:internal-conflicts} examines how changes
in conflicts within a coalition enable its members to
gain. It derives consequences for deviations from optimal
colorings and bounds the internal conflicts that can
remain after deviations from general Nash equilibria.

\section{An optimal coloring that is not a strong equilibrium}
\label{sec:counterexample}

Fix $k\ge3$. We construct a graph $G_k$ and two colorings $\sigma$
and $\tau$ that agree outside a sixteen-vertex coalition $Q$.
Every member of $Q$ gains one unit of utility under $\tau$, although
both colorings have the same cut value. The main step is to show
that this common value is optimal: the anchor graph first restricts
any competing coloring, after which a short counting argument
certifies optimality.

\subsection{Construction of the graph $G_k$}
\label{subsec:counterexample-graph}

Our construction consists of two vertex-disjoint parts, the
\emph{core graph} and the \emph{anchor graph}. The vertex set of the
core graph is the deviating coalition. The anchor graph is a complete
$k$-partite graph whose vertices play the role of outsiders: their
colors are unchanged by the deviation.

\paragraph{The core graph.}
Let
\[
  A=\{a_1,a_2,a_3,a_4\},\qquad
  B=\{b_1,b_2,b_3,b_4\},
\]
\[
  C=\{c_1,c_2,c_3,c_4\},\qquad
  D=\{d_1,d_2,d_3,d_4\}
\]
be four pairwise disjoint independent sets, and put
\[
  Q:=A\cup B\cup C\cup D.
\]
The set $Q$ will be the deviating coalition. Let $p$ be the permutation
of $\{1,2,3,4\}$ defined by
\[
  p(1)=2,\qquad p(2)=1,\qquad
  p(3)=4,\qquad p(4)=3.
\]
The core graph $G_k[Q]$ has exactly the following edges:
\begin{enumerate}
  \item $\{a_i,d_j\}$ for all $1\le i,j\le4$;
  \item $\{a_i,b_j\}$ and $\{b_i,d_j\}$ whenever $i\ne j$;
  \item $\{a_i,c_i\}$, $\{b_i,c_i\}$, and $\{b_i,c_{p(i)}\}$
        for $1\le i\le4$.
\end{enumerate}
Consequently,
\[
  |E(G_k[Q])|=16+12+12+4+4+4=52,
\]
and the core degrees of the vertices in $A,B,C,D$ are $8,8,3,7$,
respectively.

\paragraph{The anchor graph.}
For each $j\in[k]$, let
\[
  R_j=\{r_{j,1},\ldots,r_{j,25}\}
\]
be an independent set, and join every two vertices that lie in different
anchor sets. The anchors then induce the complete $k$-partite graph
\[
  K_{\underbrace{\scriptstyle 25,\ldots,25}_{k
  \text{ parts}}}.
\]

\paragraph{Edges between core and anchors.}
Each core vertex is joined to the first $m$ vertices
$r_{j,1},\ldots,r_{j,m}$ of every anchor set $R_j$, where $m$ is the
entry of the table in \Cref{fig:counterexample} for that vertex and that
set. The last column of the table applies separately to every $j$ with
$4\le j\le k$ and is absent when $k=3$. There are no further edges.
\Cref{fig:counterexample} summarizes the construction together with the
deviation defined in \Cref{subsec:counterexample-deviation}.

\paragraph{Size of the graph.}
All edges of $G_k$ have weight one. The graph is simple and connected,
and
\[
  |V(G_k)|=16+25k.
\]
The first three anchor sets contribute $128$ core--anchor edges, and
each further anchor set contributes $16\cdot6=96$. Hence
\begin{align}
  |E(G_k)|
  &=52+128+96(k-3)+625\binom{k}{2} \notag\\
  &=625\binom{k}{2}+96k-108.
  \label{eq:edge-count}
\end{align}

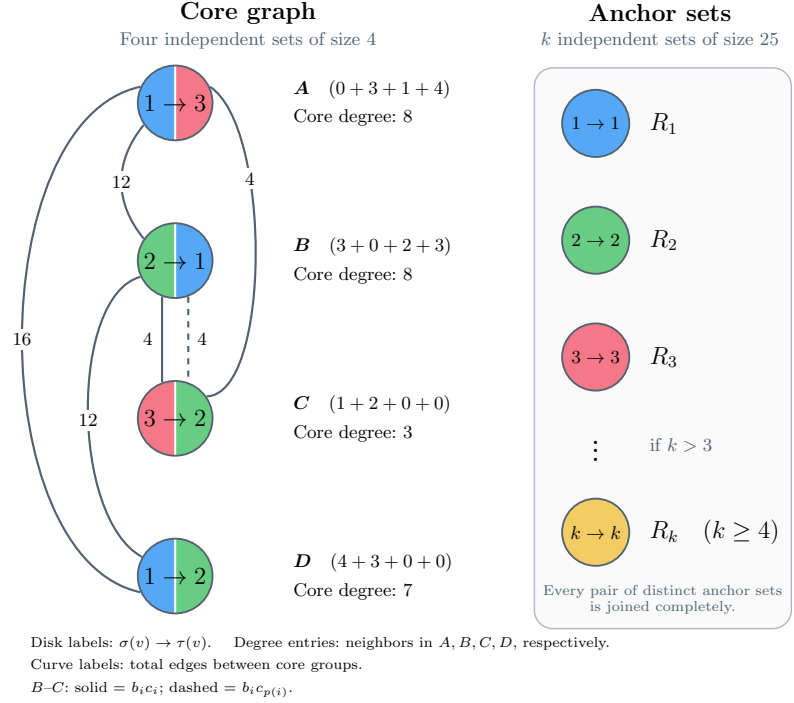
\begin{figure}[h]
\centering
\resizebox{0.82\linewidth}{!}{%
\begin{tikzpicture}[
  font=\small,
  core edge/.style={
    draw=edgeGray,line width=1.05pt
  },
  edge label/.style={
    fill=white,inner sep=2pt,font=\small
  },
  group info/.style={
    anchor=west,align=left,font=\small
  },
  anchor disk/.style={
    circle,draw=edgeGray,line width=.9pt,
    minimum size=1.15cm,inner sep=1pt
  }
]

\node[font=\bfseries\large] at (1.25,10.35)
  {Core graph};
\node[text=edgeGray] at (1.25,9.87)
  {Four independent sets of size \(4\)};

\node[font=\bfseries\large] at (8.3,10.35)
  {Anchor sets};
\node[text=edgeGray] at (8.3,9.87)
  {\(k\) independent sets of size \(25\)};

\coregroup{A}{0}{8.8}{colOne}{colThree}{1}{3}
\coregroup{B}{0}{6.1}{colTwo}{colOne}{2}{1}
\coregroup{C}{0}{3.4}{colThree}{colTwo}{3}{2}
\coregroup{D}{0}{0.7}{colOne}{colTwo}{1}{2}

\node[group info] at (1.9,8.8)
  {\(\boldsymbol{A}\quad(0+3+1+4)\)\\[3pt]
   Core degree: \(8\)};
\node[group info] at (1.9,6.1)
  {\(\boldsymbol{B}\quad(3+0+2+3)\)\\[3pt]
   Core degree: \(8\)};
\node[group info] at (1.9,3.4)
  {\(\boldsymbol{C}\quad(1+2+0+0)\)\\[3pt]
   Core degree: \(3\)};
\node[group info] at (1.9,0.7)
  {\(\boldsymbol{D}\quad(4+3+0+0)\)\\[3pt]
   Core degree: \(7\)};

\draw[core edge]
  (A.155)
  .. controls (-3.3,8.5) and (-3.3,1.0) ..
  node[edge label,pos=.50] {\(16\)}
  (D.205);

\draw[core edge]
  (A.215)
  .. controls (-1.05,7.8) and (-1.05,7.1) ..
  node[edge label,pos=.50] {\(12\)}
  (B.145);

\draw[core edge]
  (B.205)
  .. controls (-1.8,5.5) and (-1.8,1.7) ..
  node[edge label,pos=.53] {\(12\)}
  (D.150);

\draw[core edge]
  (A.25)
  .. controls (1.65,8.4) and (1.65,3.8) ..
  node[edge label,pos=.32] {\(4\)}
  (C.35);

\draw[core edge]
  (B.250) --
  node[edge label,left=2pt] {\(4\)}
  (C.110);

\draw[core edge,dashed]
  (B.290) --
  node[edge label,right=2pt] {\(4\)}
  (C.70);

\draw[
  rounded corners=9pt,
  draw=edgeGray!45,fill=edgeGray!3,line width=.7pt
] (6.15,-.2) rectangle (10.55,9.4);

\node[anchor disk,fill=colOne] at (7.2,8.45)
  {\(1\rightarrow1\)};
\node[anchor=west,font=\large] at (8.0,8.45)
  {\(R_1\)};

\node[anchor disk,fill=colTwo] at (7.2,6.45)
  {\(2\rightarrow2\)};
\node[anchor=west,font=\large] at (8.0,6.45)
  {\(R_2\)};

\node[anchor disk,fill=colThree] at (7.2,4.45)
  {\(3\rightarrow3\)};
\node[anchor=west,font=\large] at (8.0,4.45)
  {\(R_3\)};

\node[font=\Large] at (7.2,2.95) {\(\vdots\)};
\node[anchor=west,text=edgeGray] at (8.0,2.95)
  {if \(k>3\)};

\node[anchor disk,fill=colOther] at (7.2,1.45)
  {\(k\rightarrow k\)};
\node[anchor=west,font=\large] at (8.0,1.45)
  {\(R_k\quad(k\ge4)\)};

\node[align=center,font=\footnotesize,text=edgeGray]
  at (8.35,.3)
  {Every pair of distinct anchor sets\\
   is joined completely.};

\node[anchor=west,align=left,font=\footnotesize]
  at (-2.6,-.85)
  {Disk labels: \(\sigma(v)\rightarrow\tau(v)\).
   \quad Degree entries: neighbors in \(A,B,C,D\), respectively.\\[3pt]
   Curve labels: total edges between core groups.\\[3pt]
   $B$--$C$: solid = $b_ic_i$; dashed = $b_ic_{p(i)}$.};

\node[font=\bfseries] at (3.9,-1.75)
  {Number of anchor neighbors per core vertex};

\node[anchor=north,inner sep=0pt] at (3.9,-2.1) {%
  \renewcommand{\arraystretch}{1.25}%
  \setlength{\tabcolsep}{13pt}%
  \begin{tabular}{@{}lcccc@{}}
    \toprule
    Core vertex
      & \(\textcolor{colOne!65!black}{R_1}\)
      & \(\textcolor{colTwo!55!black}{R_2}\)
      & \(\textcolor{colThree!75!black}{R_3}\)
      & Each \(R_j,\ 4\le j\le k\) \\
    \midrule
    \(a_i\in A\) & \(0\) & \(6\) & \(3\) & \(6\) \\
    \(b_i\in B\) & \(0\) & \(1\) & \(6\) & \(6\) \\
    \(c_i\in C\) & \(6\) & \(0\) & \(1\) & \(6\) \\
    \(d_i\in D\) & \(0\) & \(3\) & \(6\) & \(6\) \\
    \bottomrule
  \end{tabular}%
};

\end{tikzpicture}%
}
\caption{%
The counterexample and the coalition deviation
\(\sigma\to\tau\), with colors indexed by \(1,\ldots,k\).
Each core disk represents four vertices; its left and right
halves indicate their colors before and after the deviation.
The entries beside each disk give each vertex's numbers of neighbors
in \(A,B,C,D\), respectively.
The solid and dashed $B$--$C$ edges represent the matchings
$\{b_ic_i\}$ and $\{b_ic_{p(i)}\}$, respectively.
All anchor vertices retain their colors.
The dots, yellow disk, and final table column apply only
when \(k\ge4\).
}
\label{fig:counterexample}
\end{figure}

\subsection{A profitable coalitional deviation}
\label{subsec:counterexample-deviation}

Define two colorings $\sigma$ and $\tau$ of $G_k$ as follows. Under
both colorings, every vertex of $R_j$ has color $j$, for $j\in[k]$. The
colors of the core vertices are given in \Cref{tab:coalition-payoffs},
which also lists the resulting utilities; these are derived below.
Thus exactly the vertices of $Q$ change their colors, and
$\tau(v)\neq\sigma(v)$ for every $v\in Q$.

\begin{table}[t]
\centering
\caption{Colors of the core vertices under $\sigma$ and $\tau$, and
the resulting utilities. Each row applies to every vertex of the
indicated group.}
\label{tab:coalition-payoffs}
\begin{tabular}{lcccc}
\toprule
Group & $\sigma$ & $\tau$
      & $u_v(\sigma)$ & $u_v(\tau)$ \\
\midrule
$A$ & 1 & 3 & $13+6(k-3)$ & $14+6(k-3)$ \\
$B$ & 2 & 1 & $14+6(k-3)$ & $15+6(k-3)$ \\
$C$ & 3 & 2 & $\phantom{1}9+6(k-3)$
            & $10+6(k-3)$ \\
$D$ & 1 & 2 & $12+6(k-3)$ & $13+6(k-3)$ \\
\bottomrule
\end{tabular}
\end{table}

\paragraph{Utilities of the coalition members.}
Under both colorings, the core vertices use only the colors $1,2,3$,
whereas the vertices of $R_j$ with $j\ge4$ have color $j$. Since every
core vertex has exactly six neighbors in each of $R_4,\dots,R_k$, all of
these edges are cut under both colorings. These sets therefore add
$6(k-3)$ to the utility of every core vertex, and this term cancels in
the comparison. It suffices to count the cut edges inside the core and
between the core and $R_1\cup R_2\cup R_3$.

As an example, consider $a_i\in A$. It has eight core neighbors (four in
$D$, three in $B$, one in $C$) and nine neighbors in
$R_1\cup R_2\cup R_3$. Under $\sigma$, four of its core edges and all
nine of its anchor edges are cut. Under $\tau$, all eight core edges are
cut, but only six of the nine anchor edges are. Hence
\[
  u_{a_i}(\sigma)=4+9+6(k-3)=13+6(k-3),\qquad
  u_{a_i}(\tau)=8+6+6(k-3)=14+6(k-3).
\]
The other three types of core vertices are treated in the same way;
\Cref{tab:utilities} lists the counts. In every case the utility
increases by exactly one.

\begin{table}[ht]
\centering
\caption{Cut edges and utilities of the core vertices under $\sigma$
and $\tau$. The last column omits the contribution $6(k-3)$ of
$R_4,\dots,R_k$, which is the same under both colorings.}
\label{tab:utilities}
\begin{tabular}{@{}lccc@{}}
\toprule
Vertex & Cut core edges & Cut edges to $R_1\cup R_2\cup R_3$
       & $u_v-6(k-3)$ \\
       & $\sigma\to\tau$ & $\sigma\to\tau$ & $\sigma\to\tau$ \\
\midrule
$a_i\in A$ & $4\to 8$ & $9\to 6$ & $13\to 14$ \\
$b_i\in B$ & $8\to 8$ & $6\to 7$ & $14\to 15$ \\
$c_i\in C$ & $3\to 3$ & $6\to 7$ & $9\to 10$ \\
$d_i\in D$ & $3\to 7$ & $9\to 6$ & $12\to 13$ \\
\bottomrule
\end{tabular}
\end{table}

Consequently,
\begin{equation}
  u_v(\tau)-u_v(\sigma)=1
  \qquad\text{for every }v\in Q.
  \label{eq:coalition-gain}
\end{equation}
Since $\tau$ agrees with $\sigma$ outside $Q$ and $\tau(v)\ne\sigma(v)$
for every $v\in Q$, the pair $(Q,\tau)$ is a strong deviation from
$\sigma$. Thus $\sigma$ is not a strong equilibrium.

The coalition's total gain of $16$ is exactly offset by losses among the
anchor vertices. The number of cut edges between $Q$ and $R_2$ decreases
from $36$ to $28$, and that between $Q$ and $R_3$ decreases from $60$ to
$52$. All other core--anchor edges keep their status. Since all anchor
colors remain unchanged, these decreases cause aggregate utility losses
of eight in each of $R_2$ and $R_3$, while all other anchor utilities
are unchanged. Hence the total social welfare, and therefore the cut
value, is the same under $\sigma$ and $\tau$.

Equivalently, we may count monochromatic edges. Recall that
$M(\rho)=|E(G_k)|-\mathcal C(\rho)$ for a coloring $\rho$ of $G_k$,
so maximizing the cut value is equivalent to minimizing $M$.
Table~\ref{tab:monochromatic} records the counts.

\begin{table}[ht]
\centering
\caption{Monochromatic edges under $\sigma$ and $\tau$.}
\label{tab:monochromatic}
\begin{tabular}{lcc}
\toprule
Location & $\sigma$ & $\tau$ \\
\midrule
Inside the core & 16 & 0 \\
Between the core and the anchors & 8 & 24 \\
Inside the anchor graph & 0 & 0 \\
\midrule
Total & 24 & 24 \\
\bottomrule
\end{tabular}
\end{table}

Under $\sigma$, the monochromatic core edges are the $16$ edges between
$A$ and $D$, and the monochromatic core--anchor edges are the $4$ edges
between $B$ and $R_2$ and the $4$ edges between $C$ and $R_3$. Under
$\tau$, every core edge is bichromatic, and the monochromatic
core--anchor edges are the $12$ edges between $A$ and $R_3$ and the $12$
edges between $D$ and $R_2$. In particular,
\begin{equation}
  M(\sigma)=M(\tau)=24.
  \label{eq:24-conflicts}
\end{equation}
It remains to prove that no $k$-coloring of $G_k$ has fewer than $24$
monochromatic edges.

\subsection{A certificate for optimality}
\label{subsec:anchor-rigidity}

The anchor graph turns the candidate value $M=24$ into a useful
restriction. We first show that fewer than $25$ conflicts force
the anchors to use distinct colors. We then eliminate forbidden
core colors and bound the conflicts of the remaining colorings.

\begin{lemma}[Conflict gap for complete multipartite graphs]
\label{lem:multipartite}
Let $H$ be the complete $k$-partite graph with
$n\ge1$ vertices in each part.
Every coloring $\rho\colon V(H)\to[k]$ is either
optimal or has at least $n$ monochromatic edges.
\end{lemma}

\begin{proof}
A coloring with no monochromatic edges is optimal, since assigning
a distinct color to each part cuts every edge. Suppose $\rho$ has
fewer than $n$ monochromatic edges. In each part $R_i$, at most
one vertex per monochromatic edge is incident to such an edge.
Thus there is a vertex $v_i\in R_i$ incident to no monochromatic
edge. The vertices $v_1,\ldots,v_k$ form a clique, so their colors
are distinct; relabel them so that $\rho(v_i)=i$.
Every vertex of $R_j$ is adjacent to $v_i$ for each $i\ne j$,
and therefore cannot have color $i$. Hence all vertices of
$R_j$ have color $j$, and no edge is monochromatic.
\end{proof}

Since each anchor set has $25$ vertices,
Lemma~\ref{lem:multipartite} implies that any coloring
of $G_k$ with at most $24$ monochromatic edges
has no monochromatic edges within the anchor graph.
Consequently, each anchor set is monochromatic,
and distinct anchor sets receive distinct colors.
After permuting the color labels, we may therefore
assume that every vertex in $R_j$ has color $j$.

For a core vertex $v$, call
$\{\sigma(v),\tau(v)\}$ its set of allowed colors. Explicitly, the allowed sets are
\[
  \{1,3\}\text{ on }A,\qquad
  \{1,2\}\text{ on }B,\qquad
  \{2,3\}\text{ on }C,\qquad
  \{1,2\}\text{ on }D.
\]
Every color in $[k]\setminus\{\sigma(v),\tau(v)\}$ is
\emph{forbidden} for $v$. These terms describe the proof restrictions;
all $k$ colors remain available strategies in the game.

\begin{lemma}
\label{lem:allowed-colors}
Suppose the anchor sets are optimally colored,
with $R_j$ assigned color $j$.
If a core vertex uses a forbidden color, it can
be recolored with an allowed color so as to
strictly decrease the total number of
monochromatic edges.
\end{lemma}

\begin{proof}
Let $\rho$ be the current coloring and fix a core vertex $v$ using
a forbidden color $c$. Keep all other colors fixed. Assigning
color $a$ to $v$ creates exactly $m_v(a;\rho)$ incident
monochromatic edges. Every forbidden color has six neighbors of
$v$ in the corresponding anchor set, so $m_v(c;\rho)\ge6$.
We show that at least one allowed color gives
$m_v(a;\rho)\le5$.

Let $\alpha$ and $\beta$ be the two allowed colors
of $v$.
A core neighbor can have color $\alpha$ or color
$\beta$, but not both.
Consequently, the total contribution of core
neighbors to $m_v(\alpha;\rho)+m_v(\beta;\rho)$ is at most
the core degree of $v$.
The anchor contributions are fixed by the construction.
Therefore,
\[
\begin{array}{c|c|c|c}
\text{Group}
&
\text{Core degree}
&
\text{Anchor costs}
&
m_v(\alpha;\rho)+m_v(\beta;\rho)\text{, at most}
\\ \hline
A & 8 & 0,\ 3 & 11\\
B & 8 & 1,\ 0 & 9\\
C & 3 & 1,\ 0 & 4\\
D & 7 & 0,\ 3 & 10
\end{array}
\]
Since these counts are integers,
\[
  \min\{m_v(\alpha;\rho),m_v(\beta;\rho)\}
  \le
  \left\lfloor
    \frac{m_v(\alpha;\rho)+m_v(\beta;\rho)}{2}
  \right\rfloor.
\]
Hence an allowed color yields at most $5$, $4$,
$2$, or $5$ incident monochromatic edges for a
vertex in $A$, $B$, $C$, or $D$, respectively.
In every case, this is fewer than the at least six
conflicts under its current forbidden color.

Finally, recoloring $v$ changes only edges incident
to $v$.
The strict decrease in its incident monochromatic
edges is therefore also a strict decrease in the
total number of monochromatic edges.
\end{proof}

This argument remains valid when other core
vertices use forbidden colors. Repeating the
operation produces a coloring using only
allowed colors on the core, strictly reducing $M$ at each step. Each step removes one
forbidden-colored vertex, so the process terminates. We now bound
the conflicts in the resulting restricted coloring.

\begin{lemma}\label{lem:optimality-certificate}
Suppose every vertex of $R_j$ has color $j$, and
every core vertex uses an allowed color.
Then the resulting coloring $\rho$ satisfies
\[
  M(\rho)\ge24.
\]
\end{lemma}

\begin{proof}
Write
\[
  x:=|\{a_i:\rho(a_i)=1\}|,\qquad
  y:=|\{b_i:\rho(b_i)=1\}|,
\]
and let
\[
  z:=|\{i:\rho(a_i)=\rho(b_i)=1\}|.
\]
Thus $0\le x,y\le4$.
We partition the potentially monochromatic edges
into three groups.

\medskip
\noindent
\emph{Edges between $A$ and $B$, and their anchor edges.}
The only common allowed color of $A$ and $B$ is $1$.
Since $a_i b_j$ is an edge precisely when $i\ne j$,
there are $xy-z$ monochromatic edges between these sets.
The monochromatic anchor edges incident to $A$ and $B$
number $3(4-x)$ and $4-y$, respectively.
Their combined contribution is therefore
\[
  16-3x-y+xy-z.
\]

\medskip
\noindent
\emph{Edges incident to $C$.}
First assign color $2$ to every vertex of $C$.
The resulting number of conflicts on edges incident to $C$
is $2(4-y)$: each of the $4-y$ vertices of $B$ with
color $2$ has exactly two neighbors in $C$.

Now consider recoloring $c_i$ from $2$ to $3$.
Its number of conflicts changes from the number of
color-$2$ vertices among $b_i,b_{p(i)}$ to one if
$\rho(a_i)=1$ (due to the conflict with the anchor),
and two if $\rho(a_i)=3$.
Thus assigning color $2$ to $c_i$ creates at most two
conflicts, while assigning color $3$ creates at least one.
Hence a strict reduction is possible only when both
$b_i$ and $b_{p(i)}$ have color $2$ and $\rho(a_i)=1$;
in that case, the saving is exactly one.

We now bound the total saving from recoloring vertices
of $C$. Each vertex $c_i$ can save at most one conflict,
and a saving requires, in particular,
$\rho(a_i)=1$ and $\rho(b_i)=2$.
Of the $x$ indices with $\rho(a_i)=1$, exactly $z$
have $\rho(b_i)=1$, leaving $x-z$ with $\rho(b_i)=2$.
Thus at most $x-z$ vertices of $C$ can save a conflict.

Because no two vertices of $C$ are adjacent, recoloring
one does not change the conflicts incident to another.
Starting from the $2(4-y)$ conflicts obtained by coloring
all of $C$ with color $2$, we can therefore remove at
most $x-z$ conflicts. Hence every allowed coloring of
$C$ leaves at least
\[
  2(4-y)-(x-z)=8-2y-x+z
\]
conflicts on edges incident to $C$.


\medskip
\noindent
\emph{Edges incident to $D$.}
If $\rho(b_i)=1$, assigning color $1$ or color $2$
to $d_i$ produces, respectively,
\[
  x+y-1
  \qquad\text{or}\qquad
  7-y
\]
monochromatic edges, including its anchor edges.
If $\rho(b_i)=2$, the corresponding counts are
\[
  x+y
  \qquad\text{or}\qquad
  6-y.
\]
Indeed, $d_i$ is adjacent to every vertex of $A$,
to every vertex of $B$ except $b_i$, and to
three vertices of $R_2$.

There are $y$ indices of the first type and $4-y$
of the second. Also, combining the contributions counted
above, the total number of conflicts on edges incident
to $A\cup B\cup C$ is at least
\[
  (16-3x-y+xy-z)+(8-2y-x+z)
  =24-4x-3y+xy.
\]

Hence
\begin{align*}
  M(\rho)\ge L(x,y)
  &:= 
  24+x(y-4)-3y\\
  &\quad
  +y\min\{x+y-1,\,7-y\}\\
  &\quad
  +(4-y)\min\{x+y,\,6-y\}.
\end{align*}

Since $x,y\in\{0,1,2,3,4\}$, we can evaluate this
bound directly:

\[
\begin{array}{c|ccccc}
  x\backslash y & 0 & 1 & 2 & 3 & 4\\
  \hline
  0 & 24 & 24 & 24 & 24 & 24\\
  1 & 24 & 25 & 26 & 26 & 24\\
  2 & 24 & 26 & 28 & 28 & 24\\
  3 & 24 & 27 & 28 & 27 & 24\\
  4 & 24 & 28 & 28 & 26 & 24
\end{array}
\]

Every entry is at least $24$, so $M(\rho)\ge24$.

\end{proof}

\begin{proof}[Proof of Theorem~\ref{thm:optimal-not-strong}]
Fix $k\ge3$. The graph $G_k$ constructed in \Cref{sec:counterexample} is finite,
simple, connected, and unweighted.

We first show that every coloring of $G_k$ has at least
$24$ conflicts. Suppose, to the contrary, that a coloring
$\rho$ has fewer than $24$ conflicts.
By Lemma~\ref{lem:multipartite}, the anchor sets must
receive distinct colors. Relabeling the colors, we may
assume that their colors agree with those in the construction.
We can then apply Lemma~\ref{lem:allowed-colors} repeatedly
to give every core vertex an allowed color without increasing
the number of conflicts. The resulting coloring still has
fewer than $24$ conflicts, contradicting
Lemma~\ref{lem:optimality-certificate}.

Since $M(\sigma)=M(\tau)=24$, both $\sigma$ and $\tau$
minimize the number of conflicts and hence maximize
the cut size. However, $\tau$ agrees with $\sigma$ outside $Q$, and
Equation~\eqref{eq:coalition-gain} shows that every
member of $Q$ strictly gains by deviating to $\tau$.
Thus $\sigma$ is optimal but not a strong equilibrium.
\end{proof}

\begin{remark}
The construction shows that an optimal coloring
need not be a strong equilibrium.
It does not rule out the existence of another
optimal coloring of $G_k$ that is a strong
equilibrium, and therefore leaves the general
existence question unresolved.
Also, we make no claim that the graph order or the
deviating coalition size is smallest possible.
\end{remark}

\section{Internal conflicts in improving coalitions}
\label{sec:internal-conflicts}

The preceding construction shows that an optimal coloring
need not be a strong equilibrium. Its mechanism is the
elimination of internal conflicts: the coalition removes
sixteen monochromatic edges within the core, allowing every
member to gain one unit of utility while the cut value
remains unchanged. We first formalize this mechanism through
an identity relating the coalition's total gain to the change
in its internal conflicts. This yields a necessary condition
for a strong deviation from an optimal coloring and a
sufficient condition for optimal colorings to be strong
equilibria. We then turn to deviations from general Nash
equilibria, establishing a sharp bound on the conflicts that
can remain after a deviation. In particular, the sharpness
construction also disproves the putative linear bound
$M_S(\tau)\le |S|$ for arbitrary strong deviations.
This claim must be distinguished from a bound restricted to
minimal coalitions.

Throughout this section, $G=(V,E)$ is a finite simple
unweighted graph. For a coloring $\rho$ and a set
$S\subseteq V$, write
\[
M_S(\rho)
:=
\bigl|\{\{u,v\}\in E(G[S]):\rho(u)=\rho(v)\}\bigr|,
\qquad
e(S):=|E(G[S])|.
\]
Thus $M_S(\rho)$ counts the internal conflicts of $S$, and
$M_V(\rho)=M(\rho)$. For a strong deviation $(S,\tau)$
from $\sigma$, every vertex of the nonempty coalition $S$
changes color and strictly increases its utility, while
$\tau$ agrees with $\sigma$ outside $S$.

\subsection{Internal conflicts at an optimal coloring}
\label{subsec:optimal-conflicts}

The following elementary identity makes precise the distinction
between utility gains and cut improvement. It uses the same
internal--boundary edge accounting as the cut-change inequality
in \citet[Lemma~1]{garuglieri2024optimal}; we include the exact
identity and its proof to make the role of surviving conflicts explicit.

\begin{lemma}[Deviation identity]
\label{lem:deviation-identity}
Let $\sigma$ and $\tau$ be colorings that agree outside
$S\subseteq V$. Then
\begin{equation}
\label{eq:deviation-identity}
\mathcal C(\tau)-\mathcal C(\sigma)
=
\sum_{v\in S}\bigl(u_v(\tau)-u_v(\sigma)\bigr)
-
\bigl(M_S(\sigma)-M_S(\tau)\bigr).
\end{equation}
\end{lemma}

\begin{proof}
Let $B_S(\rho)$ denote the number of bichromatic edges
with exactly one endpoint in $S$ under $\rho$. Edges
outside $S$ are unchanged, so
\[
\mathcal C(\tau)-\mathcal C(\sigma)
=
B_S(\tau)-B_S(\sigma)
+
M_S(\sigma)-M_S(\tau).
\]
The sum of the coalition's utilities counts each
bichromatic boundary edge once and each bichromatic
internal edge twice. Hence
\[
\sum_{v\in S}\bigl(u_v(\tau)-u_v(\sigma)\bigr)
=
B_S(\tau)-B_S(\sigma)
+
2\bigl(M_S(\sigma)-M_S(\tau)\bigr).
\]
Subtracting gives~\eqref{eq:deviation-identity}.
\end{proof}

\begin{theorem}[Necessary internal conflicts]
\label{thm:necessary-conflicts}
If $\sigma$ is an optimal coloring and $(S,\tau)$ is a
strong deviation from $\sigma$, then
\begin{equation}
\label{eq:necessary-conflicts}
M_S(\sigma)
\ge
M_S(\tau)+|S|
\ge
|S|.
\end{equation}
Consequently, an optimal coloring $\sigma$ is a strong
equilibrium whenever
\[
M_S(\sigma)<|S|
\qquad
\text{for every nonempty }S\subseteq V.
\]
\end{theorem}

\begin{proof}
Every deviating player gains at least one unit of utility,
since the graph is unweighted. By optimality and
Lemma~\ref{lem:deviation-identity},
\[
0
\ge
\mathcal C(\tau)-\mathcal C(\sigma)
\ge
|S|-\bigl(M_S(\sigma)-M_S(\tau)\bigr),
\]
which proves~\eqref{eq:necessary-conflicts}. The final
statement follows immediately.
\end{proof}

The optimal-coloring counterexample in the preceding
section attains~\eqref{eq:necessary-conflicts} with
equality. Indeed, its coalition $Q$ satisfies
\[
M_Q(\sigma)=16=|Q|,
\qquad
M_Q(\tau)=0,
\]
and every member gains exactly one unit of utility.
Thus the construction realizes the smallest possible
initial internal-conflict count permitted by
Theorem~\ref{thm:necessary-conflicts} for a coalition of
its size.

The next consequence is already known: it is
\citet[Proposition~9]{carosi2019coalition}. We include a short
proof to show how it follows from the internal-conflict criterion.

\begin{corollary}[Bounded maximum degree; \citealp{carosi2019coalition}]
\label{cor:bounded-degree}
If $\Delta(G)\le 2k-1$, then every optimal coloring of
the Max-$k$-Cut game on $(G,k)$ is a strong equilibrium.
In particular, this holds whenever $\Delta(G)\le 5$
and $k\ge 3$.
\end{corollary}

\begin{proof}
Let $\sigma$ be an optimal coloring, and hence a Nash
equilibrium.
At a Nash equilibrium, the current color of $v$ minimizes
the number of its neighbors assigned that color. Therefore,
\[
m_v(\sigma(v);\sigma)
\le
\left\lfloor\frac{d(v)}{k}\right\rfloor
\le 1.
\]
For every nonempty $S\subseteq V$,
\[
2M_S(\sigma)
\le
\sum_{v\in S}m_v(\sigma(v);\sigma)
\le |S|,
\]
so $M_S(\sigma)\le\lfloor |S|/2\rfloor<|S|$.
Theorem~\ref{thm:necessary-conflicts} now applies.
\end{proof}

The preceding results concern conflicts before a deviation
from an optimal coloring. We next consider the complementary
question of how many conflicts can remain after a strong
deviation from a Nash equilibrium, without assuming
optimality.

\subsection{A sharp bound after a deviation}
\label{subsec:post-conflicts}

Here the starting coloring need only be a Nash equilibrium.
The comparison with a unilateral move gives an upper bound
on the conflicts that can survive a joint deviation.

\begin{proposition}[Post-deviation conflicts]
\label{prop:post-conflicts}
Let $\sigma$ be a Nash equilibrium and $(S,\tau)$ a strong deviation.
Then
\begin{equation}
\label{eq:post-conflict-bound}
  M_S(\tau)\le
  \left\lfloor\frac{2e(S)-|S|}{4}\right\rfloor.
\end{equation}
The bound is attained for arbitrarily large coalitions, for every
fixed $k\ge2$.
\end{proposition}

\begin{proof}
For $v\in S$, let $\sigma^{v}$ be obtained from $\sigma$ by
recoloring only $v$ to $\tau(v)$. Since $\sigma$ is a Nash equilibrium,
$u_v(\sigma^{v})\le u_v(\sigma)$. Hence
\[
1\le u_v(\tau)-u_v(\sigma)
\le u_v(\tau)-u_v(\sigma^{v})
=\sum_{w\in N(v)\cap S}
 \bigl(\mathbf 1_{\{\sigma(w)=\tau(v)\}}
       -\mathbf 1_{\{\tau(w)=\tau(v)\}}\bigr).
\]
After summing over $v\in S$, each internal edge monochromatic
under $\tau$ contributes exactly $-2$: both its endpoints
change color, so neither positive indicator is one.
Every other internal edge contributes at most $2$.
Consequently,
\[
  |S|\le 2\bigl(e(S)-M_S(\tau)\bigr)-2M_S(\tau)
        =2e(S)-4M_S(\tau),
\]
which proves the bound. The following construction proves sharpness.
\end{proof}

\begin{example}[Equality and failure of a universal linear bound]
\label{ex:post-conflicts}
Fix $k\ge2$ and an integer $t\ge1$. Let $S=X\cup Y$ induce a
clique on $2t$ vertices, with $|X|=|Y|=t$.
Add disjoint anchor sets $T_1,\ldots,T_k$, each of size $t$,
and join distinct anchor sets completely. Join every vertex of
$X$ to one fixed vertex of $T_1$, and every vertex of $Y$ to
one fixed vertex of $T_2$. Join every vertex of $S$ to every
vertex of $T_j$ for $j\ge3$. There are no further edges.

Under $\sigma$, color $X$ and $T_1$ with $1$, color $Y$ and
$T_2$ with $2$, and color each remaining $T_j$ with $j$.
Each core vertex has exactly $t$ neighbors of each available
color, so no unilateral move improves its utility.
An anchor vertex has either zero or $t$ same-colored neighbors
and at least $t$ neighbors of every other color.
Thus $\sigma$ is a Nash equilibrium.

Now interchange colors $1$ and $2$ on $S$, keeping the anchors
fixed, to obtain $\tau$. Each core vertex retains $t-1$
monochromatic clique edges and loses its one monochromatic
anchor edge, so every member of $S$ gains exactly one unit.
Moreover,
\[
 |S|=2t,\qquad e(S)=\binom{2t}{2},\qquad
 M_S(\tau)=2\binom{t}{2}=t(t-1)
          =\frac{2e(S)-|S|}{4}.
\]
For $t\ge4$, this gives $M_S(\tau)>|S|$; the ratio
$M_S(\tau)/|S|=(t-1)/2$ is unbounded.
These coalitions are not minimal when $t>1$: one vertex of
$X$ and one of $Y$ can already profit by swapping colors.
Accordingly, the example concerns arbitrary strong deviations
and makes no claim about a bound for minimal coalitions.
\end{example}

\subsection{A density condition for a linear bound}
\label{subsec:density}

Although Example~\ref{ex:post-conflicts} shows that
post-deviation conflicts can grow quadratically with
coalition size, the bound becomes linear when the
induced coalition graph is sufficiently sparse.

\begin{corollary}[Density restriction]
\label{cor:density}
Let $\sigma$ be a Nash equilibrium and let $(S,\tau)$
be a strong deviation from $\sigma$. If
\[
\frac{2e(S)}{|S|}\le 5,
\]
then $M_S(\tau)\le |S|$. Conversely, if
$M_S(\tau)>|S|$, then
\[
2e(S)\ge 5|S|+4
\qquad\text{and}\qquad
|S|\ge 7.
\]
\end{corollary}

\begin{proof}
By~\eqref{eq:post-conflict-bound}, if $2e(S)\le 5|S|$,
then
\[
M_S(\tau)
\le \frac{2e(S)-|S|}{4}
\le |S|.
\]
Conversely, $M_S(\tau)>|S|$ implies
$M_S(\tau)\ge |S|+1$, and hence
\[
2e(S)\ge 4M_S(\tau)+|S|\ge 5|S|+4.
\]
Writing $s=|S|$ and using $2e(S)\le s(s-1)$ gives
$s^2-6s-4\ge 0$, so $s\ge 7$.
\end{proof}

The relevant condition is therefore the average degree
of the induced coalition graph. A violation of the linear
bound requires average degree at least $5+4/|S|$.
In particular, the bound holds for every strong deviation
on a graph whose maximum average degree is at most five.

\section{Concluding remarks}

Social optimality guarantees stability against unilateral
deviations in Max-$k$-Cut games, but our result shows that
this guarantee does not extend to coalitions, even on
simple, connected, unweighted graphs.
For every $k\ge3$, an optimal coloring can admit a joint
deviation that strictly benefits every participating player.
Moreover, the deviation can preserve social welfare:
the coalition's gains are offset by losses to players
outside it. Maximizing aggregate welfare therefore does
not, by itself, remove the incentive for coordinated
departures from a socially optimal outcome.

The internal-conflict analysis can help to identify conditions 
under which stability is restored.
The deviation identity relates the coalition's gains to
the change in cut value and the conflicts removed within
the coalition, recovering the known maximum-degree
condition for stability of optimal colorings.
For deviations from general Nash equilibria, the sharp
bound on surviving internal conflicts also shows why
coalition size alone cannot control these conflicts:
the density of the subgraph induced by the coalition
plays an essential role.

The counterexample leaves two distinct questions open:
does every unweighted Max-$k$-Cut game admit a strong
equilibrium, and does it always admit an optimal coloring
that is a strong equilibrium?
Disproving stability of every optimum does not rule out
the possibility of selecting a stable optimum.
This distinction suggests a natural direction for further
work: to determine when social optimality and coalitional
stability can be achieved simultaneously, and how such
colorings can be found.
Characterizing the graph structures that permit this
compatibility would help turn welfare maximization into
a prescription that also accounts for players' incentives
to coordinate.

\section*{Statements and Declarations}

\subsection*{Funding}
The author acknowledges financial support through a
doctoral fellowship at the Indian Institute of Technology
Bombay.

\subsection*{Competing interests}
The author declares no competing interests relevant to
this article.

\subsection*{Data and code availability}
All graph constructions and mathematical proofs supporting
the results are included in the article. The results can
be verified directly from these constructions and proofs,
without external datasets or computational code.

\subsection*{Use of AI tools}
The counterexample reported in
this note was developed in dialogue with ChatGPT (GPT-6 Astra);
all mathematical statements, proofs, and numerical checks were
subsequently verified independently by the author.



\end{document}